\documentclass[12pt,a4paper]{amsart}

\newtheorem{Result}{Result}
\newtheorem{Thm}  [Result]{Theorem}
\newtheorem{Prop} [Result]{Proposition}
\newtheorem{Lemma}[Result]{Lemma}

\def\A {{\mathcal{A}}}
\def\LL{{\mathcal{L}}}
\def\SS{{\mathcal{SS}}}
\def\SC{{\mathcal{SC}}}
\def\In{{\mathcal{I}n}}

\def\map#1#2#3{#1\colon#2\,{\longrightarrow}\,#3}

\def\overlinet#1{\,\overline{\!#1}}

\begin{document}

\title[Perturbation classes for subprojective spaces]
    {The perturbation classes problem for subprojective and superprojective
        Banach~spaces}

\author{Manuel Gonz\'alez}
\author{Javier Pello}
\author{Margot Salas-Brown}

\begin{abstract}
We show that the perturbation class for the upper semi-Fredholm operators
between two Banach spaces $X$ and~$Y$ coincides with the strictly singular
operators when $X$ is subprojective and that the perturbation class for
the lower semi-Fredholm operators coincides with the strictly cosingular
operators when $Y$ is superprojective. Similar results were proved
in~\cite{gonzalez-martinez-salas} under stronger conditions for $X$ and~$Y$.
\end{abstract}

\maketitle

\begin{section}{Introduction}

Given a class $\A$ of operators between Banach spaces, its perturbation
class $P\A$ is defined as the class of all operators~$K$ such that
$T + K \in \A$ for every $T \in \A$.
This definition is not intrinsic, in the sense that determining whether
an operator belongs to $P\A$ involves studying its behaviour with
respect to every operator in~$\A$. In this regard, it is useful
to find an intrinsic characterisation for a perturbation class $P\A$,
as its existence means that membership of an operator 
can be checked based on properties of the operator alone.

For the upper semi-Fredholm operators~$\Phi_+$, it has been long known that
strictly singular operators belong to $P\Phi_+$ \cite[Theorem~5.2]{kato};
an operator is strictly singular if its restriction to a closed
infinite-dimen\-sional subspace is never an isomorphism.
Similarly, for the lower semi-Fredholm operators~$\Phi_-$, strictly
cosingular operators belong to~$P\Phi_-$ \cite[Corollary~1]{vladimirskii};
an operator~$T$ is strictly cosingular if its composition~$QT$ with the
quotient operator~$Q$ of a closed infinite-codimensional subspace is never
a surjection.
The perturbation classes problem for the semi-Fredholm operators is the
question of whether these pairs of classes ($P\Phi_+$ and $\SS$; $P\Phi_-$
and $\SC$) coincide \cite[page~74]{gohberg-markus-feldman}
\cite[26.6.12]{pietsch} \cite[Section~3]{tylli}.
This question remained open for a time, but was eventually proved to have
a negative answer in general: there exists a separable, reflexive Banach
space~$Z$ such that $P\Phi_+(Z) \ne \SS(Z)$ and $P\Phi_-(Z^*) \ne \SC(Z^*)$
\cite{gonzalez}.

However, it is still interesting to find pairs of spaces for which the
answer to the perturbation classes problem is positive, as it means that,
at least for them, the relevant components of $P\Phi_+$ and $P\Phi_-$ do
admit an intrinsic characterisation. There are several known such cases,
including some classical results \cite{lebow-schechter} \cite{weis}.

In Theorems \ref{main+} and~\ref{main-},
we prove that the perturbation classes problem for
$\Phi_+(X,Y)$ has a positive answer when $X$ is subprojective and similarly
that the perturbation classes problem for $\Phi_-(X,Y)$ has a positive answer
when $Y$ is superprojective.
A Banach space~$X$ is subprojective if every closed infinite-dimensional
subspace of~$X$ contains an infinite-dimensional subspace complemented
in~$X$; a Banach space~$X$ is superprojective if every closed
infinite-codimensional subspace of~$X$ is contained in an
infinite-codimensional subspace complemented in~$X$. Subprojective and
superprojective spaces were introduced by Whitley to study strictly
singular and strictly cosingular operators \cite{whitley};
see~\cite{gonzalez-martinez-salas} for a fairly complete list of examples
known at the time and \cite{oikhberg-spinu} \cite{gonzalez-pello}
and \cite{galego-gonzalez-pello} for more recent discoveries.

Theorems \ref{main+} and~\ref{main-} improve on the following,
previously known results.

\begin{Thm}
\cite{lebow-schechter}
\cite{aiena-gonzalez-2002}
\label{prev+}
Let $X$ and $Y$ be Banach spaces such that $\Phi_+(X,Y)$ is not empty and
$Y$ is subprojective. Then $P\Phi_+(X,Y) = \SS(X,Y)$.
\end{Thm}

\begin{proof}
Note that $\Phi_+(X,Y) \ne \emptyset$ implies that $\SS(X,Y) \subseteq
P\Phi_+(X,Y) \subseteq \In(X,Y)$ \cite[Theorem~3.6]{aiena-gonzalez-2002}
and that $Y$ subprojective implies that\break $\SS(X,Y) = \In(X,Y)$
\cite[Theorems 4.3 and~4.4]{aiena-gonzalez-1998}.
\end{proof}

\begin{Thm}
\cite{lebow-schechter}
\cite{aiena-gonzalez-2002}
\label{prev-}
Let $X$ and $Y$ be Banach spaces such that $\Phi_-(X,Y)$ is not empty and
$X$ is superprojective. Then $P\Phi_-(X,Y) = \SC(X,Y)$.
\end{Thm}

\begin{proof}
Note that $\Phi_-(X,Y) \ne \emptyset$ implies that $\SC(X,Y) \subseteq
P\Phi_-(X,Y) \subseteq \In(X,Y)$ \cite[Theorem~3.6]{aiena-gonzalez-2002}
and that $X$ superprojective implies that $\SC(X,Y) = \In(X,Y)$
\cite[Theorems 4.3 and~4.4]{aiena-gonzalez-1998}.
\end{proof}

Theorem~\ref{main+} is stronger than Theorem~\ref{prev+} because the
hypotheses in Theorem~\ref{prev+} ($\Phi_+(X,Y)\ne\emptyset$ and
$Y$ subprojective) imply that $X$ itself is subprojective, as they mean that
a finite-codimensional subspace of~$X$ is isomorphic to a subspace of~$Y$,
and closed subspaces of subprojective spaces are subprojective too.
Similarly, Theorem~\ref{main-} is stronger than Theorem~\ref{prev-} because
the hypotheses in Theorem~\ref{prev-} ($\Phi_-(X,Y)\ne\emptyset$
and $X$ superprojective) imply that $Y$ itself is superprojective, as they
mean that a finite-codimensional subspace of~$Y$ is isomorphic to a quotient
of~$X$, and quotients of superprojective spaces are superprojective too.

Theorems \ref{main+} and~\ref{main-} also improve on similar results
obtained for spaces satisfying the formally stronger conditions of
strong subprojectivity and strong superprojectivity
introduced in~\cite{gonzalez-martinez-salas}.
A Banach space~$X$ is strongly subprojective if every closed
infinite-dimensional subspace of~$X$ contains an infinite-dimensional
subspace complemented in~$X$ with complement isomorphic to~$X$;
a Banach space~$X$ is strongly superprojective if every closed
infinite-codimensional subspace of~$X$ is contained in an
infinite-codimensional subspace complemented in~$X$ that is isomorphic
to~$X$.
Clearly, a strongly subprojective space~$X$ is subprojective, although
the question remains open as to whether there are subprojective spaces
that are not strongly subprojective, and likewise for the classes of
superprojective and strongly superprojective spaces.

\begin{Thm}
\cite[Theorems 2.6 and~3.7]{gonzalez-martinez-salas}
Let $X$ and $Y$ be Banach spaces.
\begin{itemize}
\item[(a)] If $X$ is strongly subprojective and $\Phi_+(X,Y)$ is not empty,
then $P\Phi_+(X,Y) = \SS(X,Y)$.
\item[(b)] If $Y$ is strongly superprojective and $\Phi_-(X,Y)$ is not empty,
then $P\Phi_-(X,Y) = \SC(X,Y)$.
\end{itemize}
\end{Thm}

We will use standard notation.
Given a (bounded, linear) operator $\map TXY$, 
$N(T)$ and~$R(T)$ will denote the kernel and the range of~$T$, respectively.
$\LL(X,Y)$ will stand for the set of all operators from~$X$ to~$Y$; if $\A$ is
a class of operators, then $\A(X,Y) = \A \cap \LL(X,Y)$ and $\A(X) = \A(X,X)$.
If $N$ is a closed subspace of~$X$, we will denote the induced natural
quotient operator by $\map{Q_N}{X}{X/N}$.

\end{section}

\begin{section}{Results}

We begin with a simple result that can be found in \cite[Theorem~7.21]{aiena}.
We include a proof for the convenience of the reader.

\begin{Prop}
\label{aux}
\leavevmode 
\begin{itemize}
\item[(a)] If $K\in P\Phi_+(X,Y)$ and $A\in\LL(X)$, then $KA\in P\Phi_+(X,Y)$.
\item[(b)] If $K\in P\Phi_-(X,Y)$ and $B\in\LL(Y)$, then $BK\in P\Phi_-(X,Y)$.
\end{itemize}
\end{Prop}

\begin{proof}
(a)~If $A$ is bijective, let $T\in\Phi_+(X,Y)$; then $TA^{-1}\in\Phi_+(X,Y)$,
so $T + KA = (TA^{-1} + K) A \in \Phi_+(X,Y)$, hence $KA \in P\Phi_+(X,Y)$.
For the general case, it is enough to note that $A$ can be written as the
sum of two bijective operators.

The proof of~(b) is similar.
\end{proof}

The next result was already known for $X$ strongly
subprojective~\cite{gonzalez-martinez-salas}.

\begin{Thm}
\label{main+}
Let $X$ and $Y$ be Banach spaces such that $\Phi_+(X,Y)$ is not empty and
$X$ is subprojective. Then $P\Phi_+(X,Y) = \SS(X,Y)$.
\end{Thm}

\begin{proof}
Since $\Phi_+(X,Y)$ is not empty, $Y$ must contain some closed subspace~$L$
isomorphic to a finite-codimensional subspace of~$X$; in particular,
$L$ must be subprojective.

Let $K \in \LL(X,Y) \setminus \SS(X,Y)$; we have to show that $K \notin
P\Phi_+(X,Y)$. Since $K$ is not strictly singular, there exists a closed
infinite-dimensional subspace~$U$ of~$X$ such that $K|_U$ is an isomorphism.
Considering the relative positions of the subspaces $K(U)$ and~$L$ inside~$Y$,
three cases may happen:
\begin{itemize}
\item[(a)] $K(U) \cap L$ is finite-dimensional and $K(U) + L$ is closed;
\item[(b)] $K(U) \cap L$ is infinite-dimensional;
\item[(c)] $K(U) \cap L$ is finite-dimensional and $K(U) + L$ is not closed.
\end{itemize}

(a)~If $K(U) \cap L$ is finite-dimensional and $K(U) + L$ is closed, we can
assume that $K(U) \cap L = \{0\}$ by passing to a smaller~$U$ if necessary.
Since $X$ is subprojective, we may further assume that $U$ is complemented
in~$X$, so there exists a closed subspace~$M$ of~$X$ such that
$X = U \oplus M$ and, as $M$ is an infinite-codimensional subspace of~$X$
and $L$ is isomorphic to a finite-codimensional subspace of~$X$,
there exists an isomorphic embedding $\map SML$.
Define an operator $\map{T}{X = U\oplus M}{Y}$ as $T(u+m) = -K(u) + S(m)$,
where $u\in U$ and $m\in M$; then $T\in\Phi_+(X,Y)$ but $U \subseteq
N(T+K)$, so $T+K \notin \Phi_+$, which proves that $K \notin P\Phi_+(X,Y)$.

(b)~If $K(U) \cap L$ is infinite-dimensional, we can pass to
$U \cap K^{-1}(L)$ to assume that $K(U) \subseteq L$ and, since $X$
is subprojective, we may further assume that $U$ is complemented in~$X$,
so there exists a projection $\map PXX$ with range~$U$. Now, $KP$ can be
seen as an operator $\map{KP}XL$ that is not strictly singular, where
$\Phi_+(X,L)$ is not empty and $L$ is subprojective, so $KP \notin
P\Phi_+(X,L)$ by Theorem~\ref{prev+}. As such, $KP \notin P\Phi_+(X,Y)$
and $K \notin P\Phi_+(X,Y)$ by Proposition~\ref{aux}.

(c)~If $K(U) \cap L$ is finite-dimensional and $K(U) + L$ is not closed,
there exists a compact operator $\map{K_1}{X}{Y}$ such that $(K + K_1)(U)
\cap L$ is infinite-dimensional~\cite[Theorem~2.6]{gonzalez-martinez-salas},
and then it follows that $K + K_1 \notin P\Phi_+(X,Y)$ from case~(b) and
finally that $K \notin P\Phi_+(X,Y)$.
\end{proof}

Next we recall a technical lemma.

\begin{Lemma}
\cite[Lemma~3.5]{gonzalez-martinez-salas}
\label{aux-2}
Let $K\in\LL(X,Y)$ and let $Y_0$ be a closed subspace of~$Y$ such that
$Q_{Y_0}K$ is surjective. If $E$ is a closed subspace of~$X$ such that
$K^{-1}(Y_0) \subseteq E$, then $Y$ contains a closed subspace~$F$ such that
$Y_0 \subseteq F$ and $E = K^{-1}(F)$. Moreover, if $E$ is
infinite-codimensional in~$X$, then $F$ is infinite-codimensional in~$Y$.
\end{Lemma}

The next result was already known for $Y$ strongly
superprojective~\cite{gonzalez-martinez-salas}.

\begin{Thm}
\label{main-}
Let $X$ and $Y$ be Banach spaces such that $\Phi_-(X,Y)$ is not empty and
$Y$ is superprojective. Then $P\Phi_-(X,Y) = \SC(X,Y)$.
\end{Thm}

\begin{proof}
Since $\Phi_-(X,Y)$ is not empty, $X$ must contain some closed subspace~$N$
such that $X/N$ is isomorphic to a finite-codimensional subspace of~$Y$;
in particular, $X/N$ must be superprojective.

Let $K \in \LL(X,Y) \setminus \SC(X,Y)$; we have to show that $K \notin
P\Phi_-(X,Y)$. Since $K$ is not strictly cosingular, there exists a closed,
infinite-codimen\-sional subspace $Z \subset Y$ such that $Q_Z K$ is
surjective, where $Q_Z$ is the natural quotient operator from~$Y$
onto $Y/Z$; note that this means that $R(K) + Z = Y$.
Considering the relative positions of the subspaces $K^{-1}(Z)$ and $N$
inside~$X$, three cases may happen:
\begin{itemize}
\item[(a)] $K^{-1}(Z) + N$ is finite-codimensional in~$X$, hence closed;
\item[(b)] $\overlinet{K^{-1}(Z) + N}$ is infinite-codimensional in~$X$;
\item[(c)] $\overlinet{K^{-1}(Z) + N}$ is finite-codimensional in~$X$ but
$K^{-1}(Z) + N$ is not closed.
\end{itemize}

(a)~If $K^{-1}(Z) + N$ is finite-codimensional in~$X$, hence closed, we can
assume that $K^{-1}(Z) + N = X$ by passing to a larger~$Z$ if necessary.
Since $Y$ is superprojective, we may further assume that $Z$ is complemented
in~$Y$, so there exists a projection $\map{P}{Y}{Y}$ with kernel
$N(P) = Z$, for which $R(P) = R(PK) = PK(N)$.
Also, as $Z$ is an infinite-codimensional complemented subspace of~$Y$ and
$X/N$ is isomorphic to a finite-codimensional subspace of~$Y$,
there exists a surjection $\map{S}{X/N}{Z}$, so $Z = R(SQ_N) =
SQ_N \bigl( K^{-1}(Z) \bigr)$.
Define an operator $\map{T}{X}{Y}$ as $T = SQ_N - PK$; then
$N \subseteq N(SQ_N)$ and $N(PK) = K^{-1} \bigl( N(P) \bigr) = K^{-1}(Z)$ so
$$\vcenter{\openup\jot\ialign{\hfil\strut$#$&${}#$\hfil\cr
    R(T) &= (SQ_N - PK) \bigl( K^{-1}(Z) + N \bigr) \cr
      &= SQ_N \bigl( K^{-1}(Z) \bigr) + PK(N) = Z + R(P) = Y, \cr
}}$$
hence $T$ is surjective and  $T \in \Phi_-(X,Y)$. However, $T + K =
SQ_N + (I_Y-P)K$, so $R(T+K) \subseteq Z$ and $T+K \notin \Phi_-(X,Y)$,
which proves that $K \notin P\Phi_-(X,Y)$.

(b)~If $\overlinet{K^{-1}(Z) + N}$ is infinite-codimensional in~$X$, we can
assume that $N \subseteq K^{-1}(Z)$ by passing to a larger~$Z$ if necessary
using Lemma~\ref{aux-2} and, since $Y$ is superprojective, we may further
assume that $Z$ is complemented in~$Y$, so there exists a projection
$\map{P}{Y}{Y}$ with kernel $N(P) = Z$.
As in the previous case, $R(PK) = R(P)$, so $PK \notin \SC(X,Y)$.
Furthermore, $N \subseteq K^{-1}(Z) = N(PK)$, so $PK$ factors through $X/N$
and there exists an operator $\map{T}{X/N}{Y}$ such that
$PK = T Q_N$, where $T \notin \SC(X/N,Y)$ because
$PK \notin \SC(X,Y)$ and $\SC$ is a surjective operator ideal~\cite{pietsch}.
As such, since $\Phi_-(X/N,Y)$ is not empty
and $X/N$ is superprojective, it follows that $T \notin
P\Phi_-(X/N,Y)$ by Theorem~\ref{prev-} and there exists an operator
$S \in \Phi_-(X/N,Y)$ such that $S + T \notin \Phi_-(X/N,Y)$,
for which
$$(S+T)Q_N = SQ_N + PK \notin \Phi_-(X,Y)$$
while $SQ_N \in \Phi_-(X,Y)$, so $PK \notin P\Phi_-(X,Y)$ and
$K \notin P\Phi_-(X,Y)$ by Proposition~\ref{aux}.

(c) If $\overlinet{K^{-1}(Z) + N}$ is finite-codimensional in~$X$ but
$K^{-1}(Z) + N$ is not closed, there exists a compact operator
$\map{K_1}{X}{Y}$ such that $\overline{(K + K_1)^{-1}(Z) + N}$ is
infinite-codimensional in~$X$ \cite[Theorem 3.7]{gonzalez-martinez-salas},
and then it follows that $K + K_1 \notin P\Phi_-(X,Y)$ from case~(b) and
finally that $K \notin P\Phi_-(X,Y)$.

\end{proof}

\end{section}


\begin{thebibliography}{00}

\bibitem{aiena}
P.\ Aiena.
\emph{Fredholm and local spectral theory, with applications to multipliers.}
Kluwer Acad.\ Publ., Dordrecht, 2004.

\bibitem{aiena-gonzalez-1998}
P.\ Aiena; M.\ Gonz\'alez.
\emph{On inessential and improjective operators.}
Studia Math. 131 (3) (1998) 271--287.

\bibitem{aiena-gonzalez-2002}
P.\ Aiena; M.\ Gonz\'alez.
\emph{Inessential operators between Banach spaces.}
Rend.\ Circ.\ Mat.\ Palermo (2) Suppl.\ 68 (2002) 3--26.

\bibitem{galego-gonzalez-pello}
E.~M.\ Galego; M.\ Gonz\'alez; J.\ Pello.
On subprojectivity and superprojectivity of Banach spaces.
Results Math.\ 71 (3-4) (2017) 1191--1205.

\bibitem{gohberg-markus-feldman}
I.~C.\ Gohberg; A.~S.\ Markus; I.~A.\ Feldman.
\emph{Normally solvable operators and ideals associated with them.}
Bul.\ Akad.\ \v Stiince RSS Moldoven 10 (76) (1960) 51--70.
Translation:
Amer.\ Math.\ Soc.\ Transl.\ (2) 61 (1967) 63--84.

\bibitem{gonzalez}
M.\ Gonz\'alez.
\emph{The perturbation classes problem in Fredholm theory.}
J.\ Funct.\ Anal.\ 200 (1) (2003) 65--70.

\bibitem{gonzalez-martinez-salas}
M.\ Gonz\'alez; A.\ Mart\'\i nez-Abej\'on; M.\ Salas-Brown.
\emph{Perturbation classes for semi-Fredholm operators
  on subprojective and superprojective spaces.}
Ann.\ Acad.\ Sci.\ Fenn.\ Math.\ 36 (2011) 481--491.

\bibitem{gonzalez-pello}
M.\ Gonz\'alez; J.\ Pello.
Superprojective Banach spaces.
J.\ Math.\ Anal.\ Appl.\ 437 (2) (2016) 1140--1151.

\bibitem{kato}
T.\ Kato.
\emph{Perturbation theory for nullity, deficiency and other quantities
  of linear operators.}
J.~Anal.\ Math.\ 6 (1958) 261--322.

\bibitem{lebow-schechter}
A.\ Lebow; M.\ Schechter.
\emph{Semigroups of operators and measures of noncompactness.}
J.\ Funct.\ Anal.\ 7 (1) (1971) 1--26.

\bibitem{oikhberg-spinu}
T.\ Oikhberg; E.\ Spinu.
\emph{Subprojective Banach spaces.}
J.\ Math.\ Anal.\ Appl.\ 424 (1) (2015) 613--635.

\bibitem{pietsch}
A.\ Pietsch.
\emph{Operator ideals.}
North-Holland Mathematical Library, 20.
North-Holland, 1980.

\bibitem{tylli}
H.~O.\ Tylli.
\emph{Lifting non-topological divisors of zero modulo the compact operators.}
J.\ Funct.\ Anal.\ 125 (2) (1994) 389--415.

\bibitem{vladimirskii}
Ju.\ N.\ Vladimirski\u\i.
\emph{Strictly cosingular operators.}
Soviet Math.\ Dokl.\ 8 (1967), 739--740.

\bibitem{weis}
L.\ Weis.
\emph{On perturbations of Fredholm operators in $L_p(\mu)$-spaces.}
Proc.\ Amer.\ Math.\ Soc.\ 67 (2) (1977) 287--292.

\bibitem{whitley}
R.\ J.\ Whitley.
\emph{Strictly singular operators and their conjugates.}
Trans.\ Amer.\ Math.\ Soc.\ 113 (2) (1964) 252--261.

\end{thebibliography}
\end{document}